\documentclass[12pt, a4paper]{article}
\usepackage[english]{babel}
\usepackage[iso-8859-7]{inputenc}
\usepackage{latexsym}
\usepackage{amsmath, amsthm, amssymb}
\usepackage{amssymb}
\usepackage{amsmath}
\usepackage{amsfonts}
\usepackage{amsthm}

\newtheorem{definition}{Definition}[section]
\newtheorem{theorem}{Theorem}[section]

\newtheorem{lemma}{Lemma}[section]
\newtheorem{corollary}{Corollary}[section]

\newcommand{\nn}{\mathbb{N}}
\newcommand{\cc}{\mathbb{C}}

\newcommand{\di}{\displaystyle}

\newcommand{\st}{\subset}

\begin{document}
\title{\bf Disjoint universality connected with  differential operators}

\author{V.Vlachou}         
\date{}
\maketitle 
\begin{abstract} 
In this article, we study disjoint universality for certain sequences of operators, that are connected with the differential operator. Actually, the motivation to study such sequences comes from Universal Taylor series, if you change the role of the center of convergence.
\end{abstract}

\section{Introduction }
${}$  

Universality and hypercyclicity are notions that describe the phenomenon where certain denumerable sets are dense inside corresponding metric spaces. To be more specific, assume that $X,Y$ are two topological vector spaces. A sequence of linear and continuous operators $T_n:X\to Y, \ n=1,2,\ldots$ is called hypercyclic (or universal) if there exists a vector $x_0\in X$ - called hypercyclic or universal for $(T_n)_n$- such that the orbit $\{T_nx_0: \ n\in\nn\}$ is dense in $Y$. This definition is a generalisation of  the notion of hypercyclicity of a single operator $T:X\to X$, where we consider   $T_n=T^n, \ n\in\nn$, the iterates of $T$. The notions of universality and hypercyclicity are connected with various paths of research in Functional Analysis, Ergodic theory and Dynamical systems, Complex Analysis and  Analytic Number Theory (see for example \cite{Aron}, \cite{GE}, \cite{Bayart}, \cite{steuding1}, \cite{steuding2} ).   

   In the last decade, several researchers studied the notion of disjoint hypercyclicity or universality, which deals with more than one sequence of operators. We say that two sequences of operators $T_n, S_n:X\to Y\ n=1,2,\ldots$ are disjoint hypercyclic, if there exists a vector $x_0\in X$, such that the set $\{(T_nx_0, S_nx_0): n\in\nn\}$ is dense in $Y\times Y$. As far as we know, this definition was first given in \cite{bernal} and \cite{bes}. Some more articles that deal with similar problems are  \cite{ct}, \cite{bes1}, \cite{bes2}, \cite{bes3}, \cite{sa} and \cite{vlachou1}.
   
   In the present article we deal with specific sequences of operators, which have been proved to be universal and we investigate the problem of disjoint universality. 
   To be more specific, fix a simply connected domain  $G\subset \cc$. We consider the sequence of operators  
   $\widetilde{T}_n(f): H(G)\to H(G), \ n\in\nn$:
    $$\widetilde{T}_n(f)(\zeta)=\sum_{j=0}^{n}\dfrac{f^{(j)}(\zeta)}{j!}(-\zeta)^j, \ \zeta\in G,$$
    where $H(G)$ denotes the space of holomorphic functions in $G$ endowed with the topology of uniform convergence on compacta.
    
   In \cite{bcms}, the authors presented  very elegant and original arguments and they proved that the sequence $\widetilde{T}_n$ is universal, if and only if, $0\notin G$. They were  influenced by the work of M. Siskaki in \cite{siskaki}, were the operators $(\widetilde{T_n})_n$ first appeared and the problem of universality was first posed. Actually, these operators have an interesting connection with Universal Taylor Series (we refer to \cite{siskaki} and \cite{bcms} for more details). We would also like to mention that C. Panagiotis in \cite{panagiotis} was the first to prove a result towards this direction. More specifically he proved that if $G$ is an open disk that does not contain zero, then the sequence $\widetilde{T_n} $ is universal. 
   In section 2, motivated by disjoint universality (see for example \cite{vlachou1}, \cite{bernal}, \cite{bes} )   we prove the following:
  \begin{theorem}\label{mult}
Let $G\subset \cc$ be a simply connected domain,  such that $0\notin G$. Let, in addition, $(\lambda_n^{(1)})_n$ and $(\lambda_n^{(2)})_n$ be two sequences of positive integers with $\di\lim_n\lambda_n^{(1)}=+\infty$ and $\di\lim_n\dfrac{\lambda_n^{(2)}}{\lambda_n^{(1)}}=+\infty$.  
Then the sequence of operators $D_n: H(G)\to H(G)\times H(G),$ $D_n(f)=(\widetilde{T}_{\lambda_n^{(1)}}(f), \widetilde{T}_{\lambda_n^{(2)}}(f)), \ n=1,2,\ldots$ is universal.
\end{theorem}
In order to achieve our goal, we  improve some points in the proof of Theorem 2.10 in \cite{bcms} and we  employ techniques from potential theory and, in particular, ideas presented in the proof of Bernstein-Walsh theorem (see \cite{Ransford}).

In section 3, we deal with the same problem for the more general sequences of operators 
  $T_{a,n}: H(G)\to H(G), \ n\in\nn$\ $a\in\cc\smallsetminus\{-1\}$:
    $$T_{a,n}(f)(\zeta)=\sum_{j=0}^{n}\dfrac{f^{(j)}(\zeta)}{j!}(a\zeta)^j, \ \zeta\in G.$$
As far as we know, these operators were first considered in \cite{bcms}. In Theorem 2.5 in \cite{bcms}, the authors proved that the sequence of operators $T_{a,n}$ is universal if $G\cap (a+1)G\ne \emptyset$. Surprisingly, the tools used in the case $a\ne -1$ are different and the result is not so strong. The same holds for the results of disjoint universality that we will present in this section. We use tools presented in \cite{vlachou1} on multiple universality of Taylor-type operators.
\section{Disjoint Universality for  $\widetilde{T}_n$}
${}$

In this section we prove Theorem \ref{mult}. We will start by presenting some easy lemmas, which we will use in the proof.

The first lemma, is a nice observation mentioned in \cite{bcms}.
\begin{lemma}\label{poly} Let $p$ be a polynomial of degree $d$ and let $n\geq d$. Then:
$$\widetilde{T}_n(p)=p(0).$$
\end{lemma}
The next lemma is merely a calculation, which we will use later.
\begin{lemma}\label{zf} Let $f\in H(G)$ and $n\in\nn$. Then:
$$\widetilde{T}_n(zf)(z)=(-1)^n\dfrac{f^{(n)}(z)}{n!}z^{n+1}.$$
\end{lemma}
\begin{proof}
$$\widetilde{T}_n(zf)(z)=\sum_{j=0}^{n}\dfrac{(zf)^{(j)}(z)}{j!}(-z)^j= \sum_{j=0}^{n}\dfrac{zf^{(j)}(z)+jf^{(j-1)}(z)}{j!}(-z)^j=$$
$$=z\widetilde{T}_n(f)(z)-z \widetilde{T}_{n-1}(f)(z)=z(\widetilde{T}_n(f)(z)-\widetilde{T}_{n-1}(f)(z))=z\dfrac{f^{(n)}(z)}{n!}(-z)^n.$$
\end{proof}
We continue with one more observation mentioned in \cite{bcms},  where  the image of the functions $\psi_k(z)=\frac{1}{z^k}$ under the operator $\widetilde{T}_n$ is considered. In the proof of Theorem 2.10 in \cite{bcms}  these functions were used in a very clever way. We plan to do the same.
\begin{lemma}\label{z-k} Let $k,n\in \nn$ and $\psi_k(z)=\frac{1}{z^k}$. Then:
$$\widetilde{T}_{n}(\psi_k)=\biggl[1+\sum_{j=1}^{n}\dfrac{k\cdot (k+1)\cdot \ldots\cdot (k+j-1)}{j!}\biggl]\psi_k.$$
\end{lemma}
In view of the above lemma, we will use the following notation to simplify our presentation.
\begin{definition} Let $k,n\in \nn$. We set:
$$d_{k,n}=1+\sum_{j=1}^{n}\dfrac{k\cdot (k+1)\cdot \ldots\cdot (k+j-1)}{j!}\overset{*}{=}\binom{n+k}{n}.$$
(*: Well known summation formula easy to prove with induction to n).
\end{definition}
The following lemma is yet another key-point used in \cite{bcms}.
\begin{lemma}\label{l1}Let  $(\lambda_n^{(1)})_n$  be a sequence of positive integers with \\ $\di\lim_n\lambda_n^{(1)}=+\infty$. Then $$\lim_n d_{k, \lambda_n^{(1)}}=+\infty, \ \forall k\in\nn.$$
\end{lemma}
\begin{proof} For every $k,n\in\nn$: $d_{k,n}\geq n$. The result follows.
\end{proof}
Let us see one similar result which  is connected with our effort to deal with disjoint universality.
\begin{lemma}\label{l2}Let  $(\lambda_n^{(1)})_n$ and $(\lambda_n^{(2)})_n$ be two sequences of positive integers such that $\di\lim_n\dfrac{\lambda_n^{(2)}}{\lambda_n^{(1)}}=+\infty$. Then $$\lim_n \dfrac{d_{k, \lambda_n^{(2)}}}{d_{k, \lambda_n^{(1)}}}=+\infty, \ \forall k\in\nn.$$
\end{lemma}
\begin{proof} The result follows easily from the equality:
$$\dfrac{d_{k, \lambda_n^{(2)}}}{d_{k, \lambda_n^{(1)}}}=\dfrac{(\lambda_n^{(2)}+1)(\lambda_n^{(2)}+2)\ldots(\lambda_n^{(2)}+k)}{(\lambda_n^{(1)}+1)(\lambda_n^{(1)}+2)\ldots(\lambda_n^{(1)}+k)}.$$ 
\end{proof}
At this point we would like to mention two dense subsets of $H(G)$, which will help us avoid constants that are hard to deal with.
\begin{lemma}\label{dense} Let $G\subset\cc$ be a simply connected domain, such that $0 \notin G$. \\ Then the sets $\{p(z): p \text{ polynomial with } p(0)=0\}$  and  \\ $\{p(\frac{1}{z}): p \text{ polynomial with } p(0)=0\} $ 
are dense in $H(G)$.
\end{lemma}
\begin{proof}
Let $f\in H(G)$ and let $K\subset G$ compact. Without loss of generality, we may assume that $K^c$ is connected.
Using Runge's approximation theorem (see for example \cite{rudin}) we may find a sequence of polynomials $(p_n)_n$ such that $\di\sup_{z\in K}|p_n(z)-\frac{f(z)}{z}|\xrightarrow{n\to+\infty}0.$ Then $\di\sup_{z\in K}|zp_n(z)-f(z)|\xrightarrow{n\to+\infty}0,$ so the first part of the lemma follows. \\
Similarly, if we consider a sequence of rational functions $(q_n)_n$ with pole only at 0, such that
 $\di\sup_{z\in K}|q_n(z)-zf(z)|\xrightarrow{n\to+\infty}0$, then $\di\sup_{z\in K}|\frac{q_n(z)}{z}-f(z)|\xrightarrow{n\to+\infty}0$ and the proof is complete. 
\end{proof}
Let, us give one more definition connected with universality and hypercyclicity.
\begin{definition} Assume that $X,Y$ are two topological vector spaces. A sequence $T_n: X\to Y, \ n=1,2,\ldots$ of linear and continuous operators is called transitive if for every two non-empty open sets $U\subset X$ and $V\subset Y$, there exists $n_0\in\nn$ such that $T_{n_0}(U)\cap V\ne\emptyset.$
\end{definition}
We are now ready to give the proof of our result.
\begin{theorem}
Let $G\subset \cc$ be a simply connected domain,  such that $0\notin G$. Let, in addition, $(\lambda_n^{(1)})_n$ and $(\lambda_n^{(2)})_n$ be two sequences of positive integers with $\di\lim_n\lambda_n^{(1)}=+\infty$ and $\di\lim_n\dfrac{\lambda_n^{(2)}}{\lambda_n^{(1)}}=+\infty$.  
Then the sequence of operators $D_n: H(G)\to H(G)\times H(G),$ $D_n(f)=(\widetilde{T}_{\lambda_n^{(1)}}(f), \widetilde{T}_{\lambda_n^{(2)}}(f)), \ n=1,2,\ldots$ is universal.
\end{theorem}
\begin{proof} In view of Birkhoff's transitivity theorem (see for example \cite{GE}), it suffices to prove that the sequence $(D_n)_n$ is transitive. 
So, let $U$ be an open subset of $H(G)$ and $V$ be an open subset of $H(G)\times H(G)$. Our aim is to find a function $f\in U$  such that $D_{n_0}(f)\in V,$ for some $n_0\in\nn$.\\
  In view of Lemma \ref{dense},  we may assume that:
$$U=\{f\in H(G): ||f-p||_K<\varepsilon\}$$
$$\text{ and }$$
$$V=\{f_1,f_2\in H(G)\times H(G): ||f_s-R_s||_K<\varepsilon, \ s=1,2\},$$
where $K$ is a compact subset of $G$ with connected complement, $\varepsilon>0$, $p$ is a polynomial that vanishes at zero and $R_s(z)=\frac{b_1^{(s)}}{z}+\frac{b_2^{(s)}}{z^2}+\ldots+\frac{b_q^{(s)}}{z^q},$ $ \ b_k^{(s)}\in \cc, k=1,2,\ldots q, \ s=1,2$.\\
For every $n\in\nn$, we set:
$$R_{1,n}(z)=\sum_{k=1}^{q}\dfrac{b_k^{(1)}}{d_{k,\lambda_n^{(1)}}}z^{-k} \  \text{ and } \ R_{2,n}(z)=\sum_{k=1}^{q}\dfrac{b_k^{(2)}}{d_{k,\lambda_n^{(2)}}}z^{-k}.$$
Note that in view of Lemma \ref{z-k}:
\begin{equation}\label{tr1}\widetilde{T}_{\lambda_n^{(1)}}(R_{1,n})=\sum_{k=1}^{q}\dfrac{b_k^{(1)}}{d_{k,\lambda_n^{(1)}}}\widetilde{T}_{\lambda_n^{(1)}}(\psi_k)=\sum_{k=1}^{q}\dfrac{b_k^{(1)}}{d_{k,\lambda_n^{(1)}}} d_{k,\lambda_n^{(1)}}\psi_k=R_1 \end{equation}
$$\text{ and }$$
\begin{equation}\widetilde{T}_{\lambda_n^{(2)}}(R_{2,n})=\sum_{k=1}^{q}\dfrac{b_k^{(2)}}{d_{k,\lambda_n^{(2)}}}\widetilde{T}_{\lambda_n^{(2)}}(\psi_k)=\sum_{k=1}^{q}\dfrac{b_k^{(2)}}{d_{k,\lambda_n^{(2)}}} d_{k,\lambda_n^{(2)}}\psi_k=R_2.
\end{equation}
Moreover,
$$\widetilde{T}_{\lambda_n^{(1)}}(R_{2,n})=\sum_{k=1}^{q}\dfrac{b_k^{(2)}}{d_{k,\lambda_n^{(2)}}}\widetilde{T}_{\lambda_n^{(1)}}(\psi_k)=\sum_{k=1}^{q}\dfrac{b_k^{(2)}}{d_{k,\lambda_n^{(2)}}} d_{k,\lambda_n^{(1)}}\psi_k. $$
Therefore, if $m=\min_{z\in K}|z|>0$ we have:
\begin{equation}\label{r1}
||R_{s,n}||_K\leq\sum_{k=1}^{q}\dfrac{|b_k^{(s)}|}{d_{k,\lambda_n^{(s)}}}\dfrac{1}{m^k}\xrightarrow{n\to+\infty}0, \ s=1,2
\end{equation}
\begin{equation}
||\widetilde{T}_{\lambda_n^{(1)}}(R_{2,n})||_K\leq
\sum_{k=1}^{q}\dfrac{|b_k^{(2)}|}{d_{k,\lambda_n^{(2)}}} d_{k,\lambda_n^{(1)}} \dfrac{1}{m^k}\xrightarrow{n\to+\infty}0
\end{equation}
where we have used Lemmas \ref{l1} and \ref{l2}.\\
Now, the function $R_{2,n}$ for $n$ large enough is suitable for our purposes, but $R_{1,n}$ is not. This happens because we can not control $\widetilde{T}_{\lambda_n^{(2)}}(R_{1,n})$. So we will approximate $R_{1,n}$ by polynomials using ideas presented in the proof of the theorem Bernstein-Walsh (see in \cite{Ransford}p.170).\\
Fix a closed contour $\Gamma$ in $G\smallsetminus K$ such that $\Gamma$ winds once around each point of $K$ and zero times around each point of $\cc\smallsetminus G$. Given $n\geq 2$, let $q_n$ be the Fekete polynomial of degree $n$ of $K$ (see definition 5.5.3 in \cite{Ransford}) and define:
$$p_n(z)=\dfrac{1}{2\pi i}\int_{\Gamma}\dfrac{R_{1,n}(w)}{wq_{\lambda_n^{(2)}(w)}}\cdot\dfrac{q_{\lambda_n^{(2)}}(z)-q_{\lambda_n^{(2)}}(w)}{z-w}dw, \ z\in K$$
Then $p_n$ is a polynomial of degree at most $\lambda_n^{(2)}-1$. \\
Using Cauchy's integral formula we have:
$$\dfrac{R_{1,n}(z)}{z}-p_n(z)=\dfrac{1}{2\pi i}\int_{\Gamma}\dfrac{R_{1,n}(w)}{w(w-z)}\cdot\dfrac{q_{\lambda_n^{(2)}}(z)}{q_{\lambda_n^{(2)}}(w)}dw, \ z\in K.$$
More generally, for $j\in\nn$:
$$\biggl(\dfrac{R_{1,n}(z)}{z}\biggl)^{(j)}-p^{(j)}_n(z)=\dfrac{j!}{2\pi i}\int_{\Gamma}\dfrac{R_{1,n}(w)}{w(w-z)^{j+1}}\cdot\dfrac{q_{\lambda_n^{(2)}}(z)}{q_{\lambda_n^{(2)}}(w)}dw, \ z\in K.$$
In view of the Theorem 5.5.7 in \cite{Ransford}, there exists $\theta\in (0,1)$ such that for $n$ large enough:
$$\dfrac{||q_{\lambda_n^{(2)}}||_{K}}{\di \min_{w\in\Gamma}|q_{\lambda_n^{(2)}}(w)|}\leq \theta^{\lambda_n^{(2)}}.$$ 
Now, if $\di M=\max_{z\in K}|z|$ and $d=d(\Gamma,K)>0$ for $n$ large enough  we have:
$$
||R_{1,n}-zp_n||_K\leq M \biggl|\biggl|\dfrac{R_{1,n}(z)}{z}-p_n(z)\biggl|\biggl|_K\leq M\dfrac{\ell(\Gamma)}{2\pi}\cdot\dfrac{||R_{1,n}||_{\Gamma}}{\di\min_{w\in\Gamma}|w|\ d}\cdot\theta^{\lambda_n^{(2)}}$$
Thus, using relation (\ref{r1}) we have:
\begin{equation}
||zp_n(z)||_K\leq ||R_{1,n}-zp_n||_K+||R_{1,n}||_K\xrightarrow{n\to+\infty}0.
\end{equation}
Additionally, for $n$ large enough
$$\dfrac{||(\dfrac{R_{1,n}}{z}-p_n)^{(\lambda_n^{(1)})}||_K}{\lambda_n^{(1)}!}\leq \dfrac{\ell(\Gamma)}{2\pi \ d}\cdot\dfrac{||R_{1,n}||_{\Gamma}}{\di\min_{w\in\Gamma}|w|\ d^{\lambda_n^{(1)}}}\cdot\theta^{\lambda_n^{(2)}}$$
In view of Lemma \ref{zf} we have:
$$||\widetilde{T}_{\lambda_n^{(1)}}(R_{1,n})-\widetilde{T}_{\lambda_n^{(1)}}(zp_n)||_K=|| \widetilde{T}_{\lambda_n^{(1)}}\biggl(z\cdot (\frac{R_{1,n}}{z}-p_n)\biggl) ||_K=$$
$$=||(-1)^{\lambda_n^{(1)}}\dfrac{(\frac{R_{1,n}}{z}-p_n)^{(\lambda_n^{(1)})}}{\lambda_n^{(1)}!} \cdot z^{\lambda_n^{(1)}+1}||_K\leq ||(\dfrac{R_{1,n}}{z}-p_n)^{\lambda_n^{(1)}}||_K \dfrac{M^{\lambda_n^{(1)}+1}}{{\lambda_n^{(1)}}!}\leq$$
$$\leq \dfrac{\ell(\Gamma)}{2\pi \ d}\cdot\dfrac{||R_{1,n}||_{\Gamma}}{\di\min_{w\in\Gamma}|z|\ d^{\lambda_n^{(1)}}}\cdot\theta^{\lambda_n^{(2)}} M^{\lambda_n^{(1)}+1}\xrightarrow{n\to+\infty}0$$
(Note that $\dfrac{M^{\lambda_n^{(1)}}}{d^{\lambda_n^{(1)}}}\theta^{\lambda_n^{(2)}}=\biggl(\dfrac{M^{\frac{\lambda_n^{(1)}}{\lambda_n^{(2)}}}}{d^{\frac{\lambda_n^{(1)}}{\lambda_n^{(2)}}}}\theta\biggl)^{\lambda_n^{(2)}}$).\\
Therefore in view of relation (\ref{tr1}):
\begin{equation}
||\widetilde{T}_{\lambda_n^{(1)}}(zp_n)-R_1||_K\xrightarrow{n\to+\infty}0.
\end{equation}
Moreover, for every $n$, the degree of the  polynomial $zp_n$ is  at most  $\lambda_n^{(2)}$. So taking Lemma \ref{poly} into consideration:
\begin{equation}
\widetilde{T}_{\lambda_n^{(2)}}(zp_n)=0.
\end{equation}
The same argument leads us to the conclusion that for $n$ large enough:
\begin{equation}
\widetilde{T}_{\lambda_n^{(s)}}(p)=p(0)=0, \ s=1,2.
\end{equation}
Set $f(z)=p(z)+zp_{n_0}(z)+R_{2,n_0}(z),$ for $n_0\in\nn$ large enough. \\
Then:\\
$\bullet$ $||f-p||_K\leq ||zp_{n_0}||_K+||R_{2,n_0}||_K.$\\
$\bullet$ $||\widetilde{T}_{\lambda_{n_0}^{(1)}}(f)-R_1||_K\leq ||\widetilde{T}_{\lambda_{n_0}^{(1)}}(p)||_K+
||\widetilde{T}_{\lambda_{n_0}^{(1)}}(zp_{n_0})-R_1||_K+||\widetilde{T}_{\lambda_{n_0}^{(1)}}(R_{2,n_0}||_K$\\
$\bullet$ $||\widetilde{T}_{\lambda_{n_0}^{(2)}}(f)-R_2||_K=||\widetilde{T}_{\lambda_{n_0}^{(2)}}(p)+\widetilde{T}_{\lambda_{n_0}^{(2)}}(zp_{n_0})+\widetilde{T}_{\lambda_{n_0}^{(2)}}(R_{2,n_0})-R_{2,n_0}||_K=0$\
Therefore for $n_0$ large enough $f\in U$ and $D_{n_0}(f)=(\widetilde{T}_{\lambda_{n_0}^{(1)}}(f), \widetilde{T}_{\lambda_{n_0}^{(2)}}(f))\in V$.
\end{proof}
All the steps of the previous proof apply even if we have more than two sequences of operators, thus the following theorem is also true.
\begin{theorem}\label{mult2}
Let $G\subset \cc$ be a simply connected domain,  such that $0\notin G$. Let, in addition, $(\lambda_n^{(\sigma)})_n, \ \sigma=1,2,\ldots,\sigma_0$  be a finite collection of sequences of positive integers with $\di\lim_n\lambda_n^{(1)}=+\infty$ and $\di\lim_n\dfrac{\lambda_n^{(\sigma+1)}}{\lambda_n^{(\sigma)}}=+\infty,$ $ \ \sigma=1,2,\ldots,\sigma_0-1$.  
Then the sequence of operators $D_n: H(G)\to [H(G)]^{\sigma_0},$ $D_n(f)=(\widetilde{T}_{\lambda_n^{(1)}}(f), \widetilde{T}_{\lambda_n^{(2)}}(f), \ldots,\widetilde{T}_{\lambda_n^{(\sigma_0)}}(f) ), \ n=1,2,\ldots$ is  universal.
\end{theorem}
As it turns out, Theorem \ref{mult2} gives a sufficient condition so that disjoint universality occurs. \\
\textbf{Open Question:} Are the conditions $\di\lim_n\lambda_n^{(1)}=+\infty$ and $\di\lim_n\dfrac{\lambda_n^{(\sigma+1)}}{\lambda_n^{(\sigma)}}=+\infty,$ $ \ \sigma=1,2,\ldots,\sigma_0-1$ necessary in order to obtain disjoint universality?    
\section{Disjoint Universality for $T_{a,n}$}
${}$

In this section, we are going to deal with the same problem considering the more general sequence of operators $(T_{a,n})_n$. Let us recall that if $G\subset\cc$ is a simply connected domain, we denote by   $T_{a,n}: H(G)\to H(G), \ n\in\nn$\ $a\in\cc\smallsetminus\{-1\}$:
    $$T_{a,n}(f)(\zeta)=\sum_{j=0}^{n}\dfrac{f^{(j)}(\zeta)}{j!}(a\zeta)^j, \ \zeta\in G.$$
Our result is based on a result concerning disjoint universality for universal Taylor series presented in \cite{vlachou1}. Let us start be giving some notation and a definition.

 For a compact set $K\subset\cc$, we denote
$$\mathcal{A}(K)=\{g\in C(K): \ g \text{ is
holomorphic in } K^o\}.$$
We consider the space $\mathcal{A}(K)$ with the topology of uniform convergence on $K$.\\
Moreover, 
 $$\mathcal{M}=\{K\st
\cc: K \ \text{compact set with } \ K^c \text{ connected
set}\}$$
and
 $$\mathcal{M}_{G^c}=\{K\st
G^c: K \ \text{compact set with } \ K^c \text{ connected
set}\}.$$
Finally, if $f\in H(G)$ and $\zeta\in G$, then $S_n(f,\zeta)(z)=\sum_{j=0}^{n}\dfrac{f^{j}(z)}{j!}(z-\zeta)^j,$  $ \ n=1,2,\ldots.$
\begin{definition} Let $(\lambda^{(\sigma)}_n)_n$ , $\sigma=1,2,\ldots\sigma_0$ be a finite collection of sequences of positive integers. A function $f\in H(G)$ belongs to the class 
$U_{mult}(G, [(\lambda^{(\sigma)}_n)_n]_{\sigma=1}^{\sigma_0})$, if for every choice of compact sets $K_1,K_2,\ldots, K_{\sigma_0}\in\mathcal{M}_{G^c}$ and every choice of functions $g_1\in \mathcal{A}(K_1) ,\ldots, g_{\sigma_0}\in \mathcal{A}(K_{\sigma_0})$  there exists a striclty increasing sequence of positive integers $(\mu_n)_n$ such that for every compact set $\Gamma\st \Omega$:
$$\sup_{\zeta\in\Gamma}||S_{\lambda^{(\sigma)}_{\mu_n}}(f, \zeta)-g_{\sigma}||_{K_{\sigma}}\xrightarrow{n\to+\infty}0, \ \sigma=1,2,\ldots,\sigma_0.$$
\end{definition}
We are now ready to give our result. We follow the ideas  presented in the proof of Theorem 2.5. in \cite{bcms}.
\begin{theorem}\label{tan} Let $G\subset \cc$ be a simply connected domain, $\sigma_0\in\nn$ and $a\in\cc\smallsetminus \{-1\}$. If $G\cap (a+1)G=\emptyset$,  then the sequence of linear operators $D_{n,a} H(G)\to [H(G)]^{\sigma_0}$, with
 $D_{n,a}(f)=(T_{a,n}(f),T_{a,n^2}(f),\ldots ,T_{a,n^{\sigma_0}}(f))$ is universal. 
\end{theorem}
\begin{proof} In view of Corollary 3.2 and Theorem 2.2 in \cite{vlachou1}, the class of functions  $ U_{mult}(G, [(n^{\sigma})_n]_{\sigma=1}^{\sigma_0})$ is $G_{\delta}$ and dense subset of $H(G)$, thus non empty. Let $f\in  U_{mult}(G, [(n^{\sigma})_n]_{\sigma=1}^{\sigma_0})$. We will prove that $\{D_{n,a}(f): n\in\nn\}$ is dense in $[H(G)]^{\sigma_0}$ and the proof will be complete.

 Fix a compact set $K\subset G$  and a finite collection of polynomials $p_1,p_2,\ldots, p_{\sigma_0}$.
The condition $G\cap (a+1)G=\emptyset$ assures us that the compact set $\tilde{K}=(a+1)K$ lies outside $G$. Thus there exists a sequence $(n_k)_k$ of positive integers, such that:
$$\sup_{\zeta\in K}\sup_{z\in \tilde{K}}|S_{n_k^{\sigma}}(f,\zeta)(z)-p_{\sigma}(\frac{z}{a+1})|\xrightarrow{n\to+\infty} 0, \ \sigma=1,2,\ldots,\sigma_0.$$
Therefore,
$$\sup_{\zeta\in K}|S_{n_k^{\sigma}}(f,\zeta)((a+1)\zeta)-p_{\sigma}(\zeta)|\xrightarrow{n\to+\infty} 0, \ \sigma=1,2,\ldots,\sigma_0.$$
But,
$$S_{n_k^{\sigma}}(f,\zeta)((a+1)\zeta)=\sum_{j=0}^{n_k^{\sigma}}\frac{f^{j}(\zeta)}{j!}((a+1)\zeta-\zeta)^j=T_{a,n_k^{\sigma}}(f)(\zeta)$$
and the result follows.
\end{proof}
In the above result, we worked with a specific choice of sequences of indices $(\lambda_n^{(\sigma)}), \ \sigma=1,2,\ldots,\sigma_0$. This is due to the fact that there are no other known results concerning the class  $ U_{mult}(G, [(\lambda_n^{\sigma})_n]_{\sigma=1}^{\sigma_0})$. 
Question: Can we give more examples of sequences for which disjoint universality occurs? Can we characterize all sequences for which such a phenomenon occurs?

We can give some more examples, but we have no answer for the general case. To present our arguments, we give one more definition given in \cite{vlachou1}.
\begin{definition} Let $(\lambda^{(\sigma)}_n)_n$ , $\sigma=1,2,\ldots\sigma_0$ be a finite collection of sequences of positive integers and fix a point $\zeta\in G$. A function $f\in H(G)$ belongs to the class 
$U_{mult}(G, [(\lambda^{(\sigma)}_n)_n]_{\sigma=1}^{\sigma_0}, \zeta)$, if for every choice of compact sets $K_1,K_2,\ldots, K_{\sigma_0}\in\mathcal{M}_{G^c}$ and every choice of functions $g_1\in \mathcal{A}(K_1) ,\ldots, g_{\sigma_0}\in \mathcal{A}(K_{\sigma_0})$  there exists a strictly increasing sequence of positive integers $(\mu_n)_n$ such that:
$$||S_{\lambda^{(\sigma)}_{\mu_n}}(f, \zeta)-g_{\sigma}||_{K_{\sigma}}\xrightarrow{n\to+\infty}0, \ \sigma=1,2,\ldots,\sigma_0.$$
\end{definition}
Obviously the aforementioned class is weaker than $U_{mult}(G, [(\lambda^{(\sigma)}_n)_n]_{\sigma=1}^{\sigma_0})$, since $U_{mult}(G, [(\lambda^{(\sigma)}_n)_n]_{\sigma=1}^{\sigma_0})\subset U_{mult}(G, [(\lambda^{(\sigma)}_n)_n]_{\sigma=1}^{\sigma_0}, \zeta), \ \forall \zeta\in G$. In \cite{vlachou1}, the following theorem was proved (see Theorem 2.2 in \cite{vlachou1}).
\begin{theorem}\label{main} Let $G\varsubsetneq\cc$ be a simply connected domain and let $\zeta\in G$. The class  $U_{mult}(G, [(\lambda^{(\sigma)}_n)_n]_{\sigma=1}^{\sigma_0}, \zeta)$ is non-empty, if and only if, there exists a strictly increasing sequence of positive integers $(\mu_n)_n$ such that 
 $$\lim_{n\to\infty} \lambda^{(1)}_{\mu_n}=+\infty \text{ and } \lim_{n\to\infty}\frac{\lambda^{(\sigma+1)}_{\mu_n}}{\lambda^{(\sigma)}_{\mu_n}}=+\infty, \ \ \sigma=1,2,\ldots,\sigma_0-1.$$
 \end{theorem}
We will use the aforementioned theorem, to prove the following result which is connected to the case of sequences of indices with polyonimic growth at infinity.
\begin{theorem} Let $(\lambda_n^{(\sigma)})_n$, $\sigma=1,2,\ldots, \sigma_0$ be a finite number of sequences of positive integers such that $\lim_{n}\frac{\lambda_n^{(\sigma)}}{n^{d_{\sigma}}}\in (0,+\infty)$, for $0<d_1<d_2<\ldots<d_{\sigma_0}$, then $U_{mult}(G, [(\lambda_n^{(\sigma)})_n]_{\sigma=1}^{\sigma_0})\ne\emptyset$. 
\end{theorem}
\begin{proof}
Let $(n_k)_k$ be a strictly increasing sequence of positive integers. We choose a subsequence $(q_k)_k$ of $(n_k)_k$ such that:
$$q_{k+1}>kq_k, \ \forall k\in\nn.$$
For $k$ large enough we set:
$$p_k=\biggl[\dfrac{q_k}{(\log k)^{\frac{1}{d_{\sigma_0}}}}\biggl]+1.$$
Then:
$$\dfrac{q_k}{(\log k)^{\frac{1}{d_{\sigma_0}}}}\leq p_k\leq \dfrac{q_k}{(\log k)^{\frac{1}{d_{\sigma_0}}}}+1\Rightarrow$$
$$\biggl(\dfrac{1}{(\log k)^{\frac{1}{d_{\sigma_0}}}}+\frac{1}{q_k}\biggl)^{-1} \leq\dfrac{q_k}{p_k}\leq(\log k)^{\frac{1}{d_{\sigma_0}}}$$
and
$$p_{k+1}=\biggl[\dfrac{q_{k+1}}{\log{( k+1})^{\frac{1}{d_{\sigma_0}}}}\biggl]+1\geq \dfrac{q_{k+1}}{\log{( k+1})^{\frac{1}{d_{\sigma_0}}}}\geq \dfrac{2q_{k+1}}{k}>2q_k.$$
Thus, for every $\sigma=1,2,\ldots,\sigma_0$:
\begin{align*}
\lim_k\dfrac{\lambda_{p_{k+1}}^{(\sigma)}}{\lambda_{q_k}^{(\sigma)}}=\lim_k  \dfrac{p_{k+1}^{d_{\sigma}}}{q_k^{d_{\sigma}}}=+\infty\\
\lim_k \dfrac{\lambda_{q_k}^{(\sigma)}}{\lambda_{p_k}^{(\sigma)}}=\lim_k \dfrac{q_k^{d_{\sigma}}}{p_k^{d_{\sigma}}}=+\infty \\
\end{align*}
$$
\dfrac{\lambda_{q_k}^{(\sigma)}}{\lambda_{p_k}^{(\sigma)}}=\dfrac{\frac{\lambda_{q_k}^{(\sigma)}}{q_k^{d_{\sigma}}}}{\frac{\lambda_{p_k}^{(\sigma)}}{p_k^{d_{\sigma}}}}\biggl(\dfrac{q_k}{p_k}\biggl)^{d_{\sigma}}\\
\leq 2 \biggl(\dfrac{q_k}{p_k}\biggl)^{d_{\sigma_0}}\leq k, \ k \ \ \text{large enough}.$$

Following the   proof of Theorem 2.1 in \cite{vlachou2} and Corollary 3.2 in \cite{vlachou1}, we are lead to the conclusion that  $U_{mult}(G, [(\lambda_n^{(\sigma)})_n]_{\sigma=1}^{\sigma_0})=U_{mult}(G, [(\lambda_n^{(\sigma)})_n]_{\sigma=1}^{\sigma_0}, \zeta)$, for every $\zeta\in G$. But Theorem \ref{main} implies that  $U_{mult}(G, [(\lambda_n^{(\sigma)})_n]_{\sigma=1}^{\sigma_0}, \zeta)\ne\emptyset$, for every $\zeta\in G$ and  the result follows.
\end{proof}
In view of the above result, following the proof of Theorem \ref{tan} we are lead to:
\begin{corollary}Let $G\subset \cc$ be a simply connected domain, $\sigma_0\in\nn$ and $a\in\cc\smallsetminus \{-1\}$. Let, in addition,  $(\lambda_n^{(\sigma)})_n$, $\sigma=1,2,\ldots, \sigma_0$ be a finite number of sequences of positive integers such that $\lim_{n}\frac{\lambda_n^{(\sigma)}}{n^{d_{\sigma}}}\in (0,+\infty)$, for $0<d_1<d_2<\ldots<d_{\sigma_0}$. If $G\cap (a+1)G=\emptyset$,   then the sequence of linear operators $D_{n,a} H(G)\to [H(G)]^{\sigma_0}$, with
 $D_{n,a}(f)=(T_{a,\lambda_n^{(1)}}(f),T_{a,\lambda_n^{(2)}}(f),\ldots ,T_{a,\lambda_n^{(\sigma_0)}}(f))$ is universal.  
\end{corollary}
Finally, we present one more special case which deals with sequences of indices with exponential growth at $\infty$.
\begin{theorem}Let $(\lambda_n^{(\sigma)})_n$, $\sigma=1,2,\ldots, \sigma_0$ be a finite number of sequences of positive integers such that $\lim_{n}\frac{\lambda_n^{(\sigma)}}{a_{\sigma}^{n}}\in (0,+\infty)$, for $1<a_1<a_2<\ldots<a_{\sigma_0}$, then $U_{mult}(G, [(\lambda_n^{(\sigma)})_n]_{\sigma=1}^{\sigma_0})\ne\emptyset$.
\end{theorem}
\begin{proof}Let $(n_k)_k$ be a strictly increasing sequence of positive integers. We follow ideas presented in \cite{vlachou2} and we choose a subsequence $(q_k)_k$ of $(n_k)_k$ such that:
$$q_{k+1}>q_k+2\dfrac{\log\sqrt{k+1}}{\log a_{\sigma_0}}.$$
For $k$ large enough we set:
$$p_k=q_k-\biggl[\dfrac{\log \sqrt{k}}{\log a_{\sigma_0}}\biggl].$$
Then we have:
$$\dfrac{\log\sqrt{ k}}{\log a_{\sigma_0}}-1 \leq q_k-p_k\leq \dfrac{\log \sqrt{k}}{\log a_{\sigma_0}}$$
$$\text{ and }$$
$$p_{k+1}=q_{k+1}-\biggl[\dfrac{\log \sqrt{k+1}}{\log a_{\sigma_0}}\biggl]>q_k+2 \biggl[\dfrac{\log \sqrt{k+1}}{\log a_{\sigma_0}}\biggl]-\biggl[\dfrac{\log \sqrt{k+1}}{\log a_{\sigma_0}}\biggl]\Rightarrow p_{k+1}-q_k\xrightarrow{k\to +\infty}+\infty.$$
Thus
\begin{align*}
\lim_k\dfrac{\lambda_{p_{k+1}}^{(\sigma)}}{\lambda_{q_k}^{(\sigma)}}=\lim_k a_{\sigma}^{p_{k+1}-q_k}=+\infty\\
\lim_k \dfrac{\lambda_{q_k}^{(\sigma)}}{\lambda_{p_k}^{(\sigma)}}=\lim_k a_{\sigma}^{q_k-p_k}=+\infty \\
\end{align*}
$$
\dfrac{\lambda_{q_k}^{(\sigma)}}{\lambda_{p_k}^{(\sigma)}}=\dfrac{\frac{\lambda_{q_k}^{(\sigma)}}{a_{\sigma}^{q_k}}}{\frac{\lambda_{p_k}^{(\sigma)}}{a_{\sigma}^{p_k}}}\biggl(\dfrac{a_{\sigma}^{q_k}}{a_{\sigma}^{p_k}}\biggl)\\
\leq 2 a_{\sigma_0}^{q_k-p_k}\leq 2 a_{\sigma_0}^{\frac{\log\sqrt{k}}{\log a_{\sigma_0}}}=2\sqrt{k}\leq  k, \ k \ \ \text{large enough}.$$
Following the   proof of Theorem 2.1 in \cite{vlachou2} and Corollary 3.2 in \cite{vlachou1}, we are lead to the conclusion that  $U_{mult}(G, [(\lambda_n^{(\sigma)})_n]_{\sigma=1}^{\sigma_0})=U_{mult}(G, [(\lambda_n^{(\sigma)})_n]_{\sigma=1}^{\sigma_0}, \zeta)$, for every $\zeta\in G$. But Theorem \ref{main} implies that  $U_{mult}(G, [(\lambda_n^{(\sigma)})_n]_{\sigma=1}^{\sigma_0}, \zeta)\ne\emptyset$, for every $\zeta\in G$ and  the result follows.
\end{proof}
In view of the above result, following the proof of Theorem \ref{tan} we are lead to:
\begin{corollary}Let $G\subset \cc$ be a simply connected domain, $\sigma_0\in\nn$ and $a\in\cc\smallsetminus \{-1\}$. Let, in addition,  $(\lambda_n^{(\sigma)})_n$, $\sigma=1,2,\ldots$ be a finite number of sequences of positive integers such that $\lim_{n}\frac{\lambda_n^{(\sigma)}}{a_{\sigma}^{n}}\in (0,+\infty)$, for $1<a_1<a_2<\ldots<a_{\sigma_0}$. If $G\cap (a+1)G=\emptyset$,   then the sequence of linear operators $D_{n,a} H(G)\to [H(G)]^{\sigma_0}$, with
 $D_{n,a}(f)=(T_{a,\lambda_n^{(1)}}(f),T_{a,\lambda_n^{(2)}}(f),\ldots ,T_{a,\lambda_n^{(\sigma_0)}}(f))$ is universal.  
\end{corollary}

V.Vlachou 
\\
Department of Mathematics,\\
University of Patras,\\
26500 Patras,GREECE\\
e-mail: vvlachou@math.upatras.gr

\end{document}